\documentclass[12pt,letterpaper]{amsart}
\usepackage{euler, epic,eepic,latexsym, amssymb, amscd, amsfonts, xypic, url, color, epsfig}
\input xy
\xyoption{all}

 \newlength{\baseunit}               
 \newcount{\numlines}                
\alph{subsection}

\theoremstyle{theorem}
\newtheorem*{tmnl}{Main Theorem}
\newtheorem{tm}{Theorem}

\newtheorem{lm}[tm]{Lemma}

\theoremstyle{definition}
\newtheorem{df}[tm]{Definition}
\newtheorem*{assume}{Assumption}

\theoremstyle{remark}
\newtheorem{rmk}[tm]{Remark}

\newcommand{\bbF}{\mathbf{F}}

\newcommand{\bbP}{\mathbf{P}}
\newcommand{\bbQ}{\mathbf{Q}}

\newcommand{\bbZ}{\mathbf{Z}}

\newcommand{\calL}{{ \mathcal L}}

\newcommand{\calN}{{ \mathcal N}}
\newcommand{\calO}{{ \mathcal O}}

\newcommand{\Spec}{ {\operatorname{Spec}}}

\begin{document}
\pagestyle{plain}
\title{An explicit semi-factorial compactification of the N\'{e}ron model}

\address[Current]{Dept.~of Mathematics, University of South Carolina, Columbia~SC}
\address[Former]{Leibniz Universit\"{a}t Hannover, Institut f\"{u}r algebraische Geometrie, Welfengarten 1, 30060 Hannover, Germany}
\email{kassj@math.sc.edu}


\author{Jesse Leo Kass}

\begin{abstract}

C.~P\'{e}pin recently constructed a semi-factorial compactification of the  N\'{e}ron model of an abelian variety using the flattening technique of Raynaud--Gruson.  Here we prove that an explicit semi-factorial compactification is a certain moduli space of sheaves --- the family of compactified jacobians.

\vskip 0.5\baselineskip

\noindent{\bf R\'esum\'e} \vskip 0.5\baselineskip \noindent
{\bf Une compactification semi-factorielle explicite du mod\`{e}le de N\'{e}ron. }  C.~P\'{e}pin a  construit r\'{e}cemment  une compactification semi-factorielle du mod\`{e}le de N\'{e}ron d'une vari\'{e}t\'{e} ab\'{e}lienne en utilisant les techniques de platification de Raynaud--Gruson. Ici, nous montrons qu'une compactification semi-factoriel explicite d'un certain espace de modules de faisceaux --- la famille de jacobiennes compacifi\'{e}es.

\end{abstract}

\maketitle

{\parskip=12pt 

We prove that the family of compactified jacobians is a semi-factorial compactification of the N\'{e}ron model of the jacobian.  Semi-factoriality is a weakening of factoriality, the condition that the local rings are unique factorization domains.  In \cite{pepin13}, P\'{e}pin introduced the condition and proved that the N\'{e}ron model of an abelian variety $A_{K}$  over the field of fractions $K$ of a discrete valuation ring $R$ admits a semi-factorial compactification.

P\'{e}pin constructed the compactification using the flattening technique of Raynaud--Gruson \cite{raynaud71}.  We give an alternative construction when $A_{K}=J_{K}$ is a jacobian satisfying suitable hypotheses.  We prove that an explicit semi-factorial compactification is given by a compactification of $J_{K}$ as a moduli space --- by the family of compactified jacobians.

What is the compactified jacobian?  Suppose $A_{K} = J_{K}$ is the jacobian of the smooth curve $X_{K}$.  The curve $X_{K}$ extends to a regular model $X/ S$ over $S = \Spec(R)$.  The jacobian $J_{K}$ is the moduli space of degree $0$ line bundles on $X_{K}$, and we can try to extend it to a family $\overline{J}/S$ by adding over the point $0 \in S$ a moduli space of sheaves on $X_{0}$.  When $X_0$ is geometrically integral, we can extend $J_{K}$ by adding the moduli space of degree $0$ rank $1$, torsion-free sheaves on $X_0$, and this extended family is the family of compactified jacobians.   

The line bundle locus $J/S$ in a family of compactified jacobians $\overline{J}/S$ is canonically isomorphic to the N\'{e}ron model of $J_{K}$ by (a special case of)  \cite[Theorem~3.9]{kass13}, a result that extends earlier work on the topic  \cite{oda79, busonero08, caporaso08b, caporaso08a,  caporaso12, melo12a}.  Compactified jacobians are proper by construction, so $\overline{J}/S$ is a compactification of the N\'{e}ron model.  When the Picard rank of $J_K$ is $1$,  $\overline{J}/S$  has the desirable properties studied by P\'{e}pin:

\begin{tmnl} \label{Theorem: MainTheorem}
	The Altman--D'Souza--Kleiman family of compactified jacobians $\overline{J}/S$ is a semi-factorial model of the N\'{e}ron model provided the Picard rank of $J_K$ is $1$.
\end{tmnl}

The Main Theorem includes the explicit hypothesis that $J_{K}$ has Picard rank $1$ and the implicit hypothesis that the special fiber $X_0$ is geometrically integral.  How are these hypotheses used?  When the hypothesis that $X_0$ is geometrically integral fails, the Altman--D'Souza--Kleiman family $\overline{J}/S$ is not defined because the moduli space of degree $0$ rank $1$, torsion-free sheaves on $X_0$ is badly behaved.  A well-behaved space can be recovered by imposing e.g.~a stability condition, but the proof we give here does not immediately apply to these more general spaces.  In proving the Main Theorem, we use the property that translation $\tau_{a_K} \colon J_K \to J_K$ by a point $a_K \in J_{K}(K)$ extends to an automorphism $\overline{J} \to \overline{J}$.  It is not known if  $\overline{J}$ has this extension property when $X_0$ is reducible; the issue is that, when $X_0$ is reducible, the tensor product of two slope semi-stable line bundles can fail to be semi-stable.

The hypothesis that $J_K$ has Picard rank $1$ is used to assert that the N\'{e}ron--Severi group $\operatorname{NS}(J_{\overline{K}})$ is generated by classes that extend to $\overline{J}$.  Under the rank $1$ hypothesis, $\operatorname{NS}(J_{\overline{K}})$ is generated by the class of the theta divisor, and Esteves and Soucaris have (independently) shown that this divisor extends.  In general, when $\operatorname{NS}(J_{\overline{K}})$ is generated by classes that extend, our proof shows that $\overline{J}$ is semi-factorial, and it would be desirable to have more general results describing when classes in $\operatorname{NS}(J_{\overline{K}})$ extend.

\subsection{Preliminaries}
Here we collect results from the literature.  Fix a discrete valuation ring (or dvr for short) $R$ with field of fractions $K$ and residue field $k(0)$.  Set $S=\Spec(R)$ and $0=\Spec(k(0))$.  We fix a smooth \textbf{curve} $X_K/\Spec(K)$ (i.e.~a $K$-scheme of pure dimension 1 that is proper, smooth,  and geometrically connected over $K$) that we assume has genus $g \ge 1$ and study the associated  \textbf{jacobian} $J_K/\Spec(K)$.  The jacobian is a g-dimensional abelian variety that represents the \'{e}tale sheaf parameterizing degree $0$ line bundles on $X_K$, and it extends to the N\'{e}ron model $J/S$, a certain (possibly nonproper) $S$-scheme.  We omit the definition, but one consequence, which we will use, is that the restriction map $J(S) \to J_{K}(K)$ is surjective, i.e.~the weak N\'{e}ron Mapping Property holds.

To study compactifications of $J_{K}$, we make the following definitions.
\begin{df}
	A $S$-scheme $V/S$ is \textbf{semi-factorial} if the restriction map
	\begin{equation}
		\operatorname{Pic}(V) \to \operatorname{Pic}(V_K)
	\end{equation}
	on Picard groups is surjective. 
	
	If $V/S$ is separated and of finite type over $S$, then an  \textbf{$S$-compactification} of $V/S$ is a proper $S$-scheme $\overline{V}/S$ and a $S$-immersion $V \to \overline{V}$ with dense image.  An $S$-compactification  is a \textbf{semi-factorial model} if $\overline{V}/S$ is flat and projective over $S$, normal, and semi-factorial.  A semi-factorial model is a \textbf{regular model} if $\overline{V}$ is a regular scheme.
\end{df}
Corollaire~6.4 of \cite{pepin13} states that the N\'{e}ron model $J/S$ admits a semi-factorial model.  In fact, the Corollaire states that the semi-factorial model can be chosen to have certain desirable base-change properties, which we discuss in Remark~\ref{Remark: Extending}.

The curve $X_K$ admits a regular model $X/S$ because resolution of singularities holds for arithmetic surfaces \cite{lipman}. (Lipman's result is stated for $R$ excellent, but the argument on \cite[page~87]{deligne69} shows  that this hypothesis can be removed.)  For the remainder of this paper, we fix a regular model $X$ satisfying

\begin{assume}
 	$X/S$ is a regular model of $X_K$ with geometrically integral special fiber.
\end{assume}

With this assumption, the Altman--D'Souza--Kleiman \textbf{family of compactified jacobians} $\overline{J}/S$ associated to $X/S$ is defined.  The family of compactified jacobians is an $S$-scheme $\overline{J}/S$ that is projective over $S$ and represents the \'{e}tale sheaf parameterizing families of degree $0$ rank $1$, torsion-free sheaves on $X/S$ \cite[(8.10)~Theorem]{altman80}.  (Under more restrictive hypotheses, this is \cite[Theorem~II.4.1]{dsouza79}.)  The line bundle locus in $\overline{J}$  is  an open subscheme $J$ that is the N\'{e}ron model of $J_{K}$ \cite[Theorem~3.9]{kass13}.

We now recall the definition of the N\'{e}ron-Severi group and the Picard scheme of  $J_K$.  The \textbf{Picard scheme} $\operatorname{Pic}(J_K/K)/\Spec(K)$ is a $K$-group scheme that is locally of finite type over $K$ and represents the \'{e}tale sheaf parameterizing line bundles on $J_K$.  The line bundles that are algebraically equivalent to zero are parameterized by the identity component $\operatorname{Pic}^{0}(J_K/K)$ of the Picard scheme, which is an open and closed $K$-subgroup scheme that is of finite type over $K$.  

Algebraic equivalence classes of line bundles on $J_K$ form the  \textbf{N\'{e}ron--Severi group} which is defined as
$$
	\operatorname{NS}(J_{\overline{K}}) := \frac{\operatorname{Pic}(J_{K}/K)(\overline{K})}{\operatorname{Pic}^{0}(J_{K}/K)(\overline{K})}
$$
for $\overline{K}$ a fixed algebraic closure of $K$.  This group is  finitely generated, hence has a well-defined rank called the \textbf{Picard rank}. 

The Picard rank of $J_K$ is at least $1$.  Indeed, $J_{\overline{K}}$ admits a special type of divisor: the classical theta divisor.  If $\calN_{\overline{K}}$ is a line bundle of degree $g-1$ on $X_{\overline{K}}$, then 
$$
	\Theta_{\overline{K}} := \{ [\calL_{\overline{K}}] \colon h^{0}(X_{\overline{K}}, \calL_{\overline{K}} \otimes \calN_{\overline{K}}) \ne 0 \} \subset J_{\overline{K}} 
$$
is an ample divisor that defines a principal polarization.  That is, the homomorphism
\begin{gather}
	\phi \colon J_{\overline{K}} \to \operatorname{Pic}^{0}(J_{\overline{K}}/\overline{K})  \text{ defined by } \label{Eqn: Polarization} \\
	\phi(a) = \calO_{J_{\overline{K}}}(\tau_{a}^{*}(\Theta_{\overline{K}})-\Theta_{\overline{K}}) \notag
\end{gather}
is an isomorphism.  Here $\tau_{a}$ is the translation-by-$a$ map.

The divisor $\Theta_{\overline{K}}$ depends on the choice of $\calN_{\overline{K}}$, but its image in the N\'{e}ron--Severi group is independent of the choice, and we denote this common image by $\theta$.  Because $\Theta_{\overline{K}}$ is a principal polarization, $\theta$ is nonzero, and furthermore:
\begin{lm} \label{Lemma: ThetaGenerates}
	The class $\theta$ freely generates $\operatorname{NS}(J_{\overline{K}})$ when the Picard rank of  $J_{K}$ is $1$.
\end{lm}
\begin{proof}
	If $J_{K}$ has Picard rank $1$, then the N\'{e}ron--Severi group $\operatorname{NS}(J_{\overline{K}})$ is cyclic because it is torsion-free \cite[Corollary~2, page~178]{mumford70}, so we may fix a generator  $c$.  Writting $\theta = n \cdot c$ for some $n \in \bbZ$, we have
	\begin{align*}
		n^g \cdot (c^g/g!) 	= & \theta^{g}/g! \\
						=& 1 \text{ by the Riemann--Roch Formula.}
	\end{align*}
	So $n^g$ divides $1$ and hence $n=\pm 1$.  	
\end{proof}

\subsection{Proof of the Main Theorem} \label{Section: semi-factorial}
Here we prove that $\overline{J}/S$ is a semi-factorial model of the N\'{e}ron model provided the Picard rank of $J_{K}$ is $1$.

\begin{lm} \label{Lemma: FlatEct}
 $\overline{J} \to S$ is flat, and  $\overline{J}$  is Cohen--Macaulay and normal.  
\end{lm}
\begin{proof}
	Theorem~(9) of \cite{kleiman77} states that $\overline{J} \to S$ is flat with Cohen--Macaulay fibers. (That theorem includes the hypothesis that $X$ lies on a $S$-smooth family of surfaces, but we can reduce to this case by arguing as in the proof of  \cite[Lemma~3.4]{esteves02}.)   Since $S$ is Cohen--Macaulay, we can conclude that $\overline{J}$ itself is Cohen--Macaulay.

	We prove  $\overline{J}$ is normal using Serre's criteria.  To verify the criteria, we  need to show that Condition~R1 holds. The line bundle locus $J_0 \subset \overline{J}_0$ is dense in the special fiber by \cite[Theorem~(9)]{kleiman77}, so  the line bundle locus $J \subset \overline{J}$ in the total space contains all codimension 1 points.  The  locus $J$ is contained in the smooth locus of $\overline{J}/S$, hence in the regular locus of $\overline{J}$, and so Condition~R1 is satisfied.
\end{proof}

\begin{proof}[Proof of the Main Theorem]
	By Lemma~\ref{Lemma: FlatEct} we just need to show that $\overline{J}/S$ is semi-factorial, i.e.~
	\begin{equation} \label{Eqn: SemifactorialHomo}
		\operatorname{Pic}(\overline{J}) \to \operatorname{Pic}(J_K) 
	\end{equation}
	is surjective.

	First, assume that $X$ admits a line bundle $\calN$ with fiber-wise degree $g-1$.  Then the set $\{ [\calL] \in \overline{J} \colon h^{0}(X, \calL \otimes \calN ) \ne 0 \} \subset \overline{J}$ is the support of a relatively effective divisor $\Theta$  that extends the classical theta divisor by \cite[Theorem~13]{soucaris} (or \cite[page~184]{esteves97}).  In particular, $\calO_{J_{K}}(\Theta_{K})$ lies in the image of \eqref{Eqn: SemifactorialHomo}.  
	
	That image also contains all line bundles algebraically equivalent to zero.  Indeed, the polarization isomorphism $\phi$ from Equation~\eqref{Eqn: Polarization} is defined over $K$, so if $\calL_{K}$ is a line bundle on $J_K$ that is algebraically equivalent to zero, then we can write $[\calL_{K}] = \phi(a_{K})$ for some $a_{K} \in J_{K}(K)$.  Here $[\calL_K] \in \operatorname{Pic}^{0}(J_{K}/K)(K)$ is the point represented by $\calL_{K}$.  The $S$-scheme $J/S$ satisfies the N\'{e}ron Mapping Property (by e.g.~\cite[Theorem~3.9]{kass13}), so $a_{K} \in J_{K}(K)$ is the restriction of some $a \in J(S)$.  The line bundle locus $J$ acts on $\overline{J}$ (by tensor product), so translation $\tau_{a} \colon \overline{J} \to \overline{J}$ by $a$ is well-defined, and the line bundle $\calL:= \calO_{\overline{J}}( \tau_{a}^{*}(\Theta) - \Theta)$  extends $\calL_{K}$.  
	
	We have now shown that the image of \eqref{Eqn: SemifactorialHomo} contains both $\calO_{J_{K}}(\Theta_{K})$ and the line bundles algebraically equivalent to zero.  Together these line bundles generate $\operatorname{Pic}(J_{K})$ by Lemma~\ref{Lemma: ThetaGenerates}, so \eqref{Eqn: SemifactorialHomo} is surjective, proving the theorem in the special case that a $\calN$ exists.

In the general case,  we argue as allows.  Given a line bundle $\calL_K$ on $J_K$, we can extend $\calL_K$ to a family $\calL$ of rank $1$, torsion-free sheaves on $\overline{J}$ (by e.g.~the $S$-projectivity of the relevant compactified Picard scheme).  There exists a line bundle $\calN$ with fiber-wise degree $g-1$ on $X_T$ for some \'{e}tale cover $T \to S$ with $T$ the spectrum of a dvr because $X_0$ is geometrically reduced.  Say $L$ is the field of fractions of the dvr $\Gamma(T, \calO_{T})$.  The base-change $X_T$ remains regular, so $\calL_L$ extends to a line bundle on $\overline{J}_T$.  This extension must equal $\calL_T$ (by e.g.~the $S$-separatedness of the relevant compactified Picard scheme), so $\calL_{T}$ and hence $\calL$ must be a line bundle. 
\end{proof}

\begin{rmk} \label{Remark: NonFactorial}
	Does $\overline{J}$ satisfy stronger conditions than semi-factoriality?  Typically $\overline{J}$ does not satisfy the condition of regularity.  Let $K=\bbQ$, $R=\bbZ_{(3)}$ (the localization of $\bbZ$ at $3$),  $S = \Spec(R)$, and $X/S$ the minimal proper regular model of the affine curve $\Spec(R[x,y]/(y^2-x^2 (x-1)^2 (x^2+1)-3)$.  The family $X/S$ is a family of genus 2 curves with special fiber $X_0$ a rational curve with 2 nodes.  Consider the family of compactified jacobians $\overline{J}/S$ associated to $X/S$.

	   If $\nu \colon \bbP^1 \cong \widetilde{X}_0 \to X_0$ is the normalization, then $\overline{J}$ has a singularity at the rank $1$, torsion-free sheaf $I := \nu_{*} \calO(-2)$.  The singularity of $\overline{J}$ at $I$ is computed in  \cite{kass09}.  The sheaf $I$ fails to be locally free at 2 nodes, so by \cite[Lemma~6.2]{kass09} the completed local ring  is isomorphic to
	\begin{displaymath}
		\widehat{O}_{\overline{J}, [I]} = \widehat{R}[[a_1, b_1, a_2, b_2]]/(a_1 b_1 - 3, a_2 b_2 - 3)
	\end{displaymath}
	This ring not only fails to be regular, but it also fails to be factorial. (The height 1 prime $(a_1, a_2)$ is nonprincipal because the images of $a_1$, $a_2$ in the quotient $(3, a_i, b_i)/(3, a_i, b_i)^{2}$ are linearly independent.)
	
	However,  $\overline{J}/S$ is semi-factorial.  Indeed, by the Main Theorem, we just need to show that $J_{K}=J_{\bbQ}$ has Picard rank $1$, and we do so as follows. The N\'{e}ron--Severi group $\operatorname{NS}(J_{\overline{\bbQ}})$ injects  into the endomorphism ring $\operatorname{End}(J_{\overline{\bbQ}})$, and we compute this endomorphism ring by relating it to the endomorphism ring of the reduction of $J_{\bbQ}$ at a prime of good reduction.

	Both the curve $X_{\bbQ}$ and its jacobian $J_{\bbQ}$ have good reduction at the primes $p=5, 13$, as can be seen by reducing the equation $y^2=x^2(x-1)^2 (x^2+1)+3$ mod $p$.  Using this equation to naively count $\bbF_{p^n}$-points, we compute that the characteristic polynomial $f_p$ of the Frobenius endomorphism of $J_{\bbF_{p}}$ is 
	\begin{gather*}
		f_5 = x^4 - 2 x^3 + 3 x^2 - 10 x + 25, \\
		f_{13} = x^4 + 7 x^3 + 35 x^2 + 91 x + 169.
	\end{gather*}
	Applying   \cite[Theorem~6]{howe02} to these polynomials, we get that the reduction $J_{\bbF_{p}}$ is absolutely simple for $p=5, 13$, so  $\bbQ \otimes \operatorname{End}(J_{\overline{\bbF}_{p}}) = \bbQ[x]/(f_p)$.

	The reduction map injects $\bbQ \otimes \operatorname{End}(J_{\overline{\bbQ}})$ into $\bbQ \otimes \operatorname{End}(J_{\overline{\bbF}_{p}})$ for $p=5, 13$.  A computation shows that the discriminant of $\bbQ[x]/(f_5)$ is coprime to the discriminant of $\bbQ[x]/(f_{13})$, and $\bbQ$ has no nontrivial unramified extensions, so the only field contained in both $\bbQ[x]/(f_5)$ and $\bbQ[x]/(f_{13})$ is $\bbQ$.  In particular, $\operatorname{End}(J_{\overline{\bbQ}})=\bbZ$. 	This example was suggested to the author by Bjorn Poonen.

\end{rmk}

\begin{rmk} \label{Remark: Extending}
	Corollaire~6.4 of  \cite{pepin13} proves that a semi-factorial model $\widetilde{J}/S$ of $J_{K}$ can be chosen to be well-behaved with respect to certain dvr extensions.  To be precise, given morphisms $T_1 \to S$, \dots, $T_n \to S$ corresponding to extensions of $R$ contained  in the strict henselization $R^{\text{hs}}$, a semi-factorial model $\widetilde{J}/S$ can be chosen so that $\widetilde{J}_{T}/T$ is a semi-factorial model when $T \to S$ equals either some $T_i \to S$ or a morphism corresponding to a ``\emph{permise}" dvr extension.
	
	The family $\overline{J}/S$ of compactified jacobians satisfies this condition.  In fact, it satisfies a stronger condition.  By definition the formation of the family of compactified jacobians commutes with arbitrary base change, so if $T \to S$ is a morphism that corresponds to a dvr extension, then $\overline{J}_{T}/T$ is a semi-factorial model of the N\'{e}ron model provided $X_T$ is regular.  The scheme $X_T$ is regular when $T \to S$ is one of the morphisms considered by P\'{e}pin or more generally when $T \to S$ is regular and surjective (see \cite[Remarque~5.5]{pepin13}).
\end{rmk}

\subsection*{Acknowledgements} The author thanks the anonymous referee for helpful feedback and, in particular, for suggesting the proof of the Main Theorem, which simplifies an earlier argument of the author.  The author thanks Bjorn Poonen for suggesting the example in Remark~\ref{Remark: NonFactorial}; Davesh Maulik and Damiano Testa for  discussions about jacobians of N\'{e}ron--Severi rank $1$; Michael Filaseta for a discussion about field theory; David Harvey for verifying the computations in Remark~\ref{Remark: NonFactorial} using  the {M}agma algebra system;  Dino Lorenzini for feedback on an earlier draft of this article; Ethan Cotterill for help with the French language.

This work was completed while the author was a Wissenschaftlicher Mitarbeiter at the Institut f\"{u}r Algebraische Geometrie, Leibniz Universit\"{a}t Hannover. During that time, the author was supported by an AMS-Simons Travel Grant.

} 

\bibliographystyle{AJPD}



\providecommand{\bysame}{\leavevmode\hbox to3em{\hrulefill}\thinspace}
\providecommand{\MR}{\relax\ifhmode\unskip\space\fi MR }
\providecommand{\MRhref}[2]{%
  \href{http://www.ams.org/mathscinet-getitem?mr=#1}{#2}
}
\providecommand{\href}[2]{#2}

\end{document}